\RequirePackage{fix-cm}
\documentclass[smallextended]{svjour3}       
\smartqed  
\usepackage{graphicx}
\usepackage{amssymb}
\usepackage{amsmath}
\usepackage{marvosym}
\usepackage{array}
\usepackage{booktabs}
\usepackage{multirow}
\usepackage{algorithm}
\usepackage{color}
\usepackage{xcolor}
\usepackage{tabu}
\usepackage{tabularx}
\usepackage[noend]{algpseudocode}
\usepackage{epstopdf}

\newcommand{\ds}{\displaystyle}

\newcommand{\mbf}{\mathbf}

\usepackage{amssymb}
\usepackage{mathrsfs}
\begin{document}

\title{A fast method for variable-order space-fractional diffusion equations
}

\titlerunning{Fast collocation method for FDEs}        

\author{Jinhong Jia \and Xiangcheng Zheng \and Hong Wang }

\authorrunning{Jia, Zheng and Wang} 

\institute{Jinhong Jia \at
              School of Mathematics and Staticstics, Shandong Normal University, Jinan, Shandong 250358, China \\
              \email{jhjia@sdnu.edu.cn}           
           \and
           Xiangcheng Zheng and Hong Wang \at
             Dempartment of Mathematics, University of South Carolina, Columbia, South Carolina 29208, USA\\
             \email{xz3@math.sc.edu and hwang@math.sc.edu}
}

\date{Received: date / Accepted: date}

\maketitle

\begin{abstract}
We develop a fast divided-and-conquer indirect collocation method for the homogeneous Dirichlet boundary value problem of variable-order space-fractional diffusion equations. Due to the impact of the space-dependent variable order, the resulting stiffness matrix of the numerical approximation does not have a Toeplitz-like structure. In this paper we derive a fast approximation of the coefficient matrix by the means of a sum of Toeplitz matrices multiplied by diagonal matrices. We show that the approximation is asymptotically consistent with the original problem, which requires $O(kN\log^2 N)$ memory and $O(k N\log^3 N)$ computational complexity with $N$ and $k$ being the numbers of unknowns and the approximants, respectively. Numerical experiments are presented to demonstrate the effectiveness and the efficiency of the proposed method.
\end{abstract}

\section{Introduction}

Field tests showed that space-fractional diffusion equations (sFDEs) provide more accurate descriptions of challenging phenomena of superdiffusive transport and long range interaction, which occur in solute transport in heterogeneous porous media and other applications, than integer-order diffusion equations do \cite{BenSchMee,DelCar,MetKla04,SchBenMee01}. In fact, integer-order diffusion equations were derived if the underlying independent and identically distributed particle movements have (i)  a mean free path and (ii) a mean waiting time. In this case, the central limit theorem concludes that the (normalized) partial sum of the independent and identically distributed particle movements converges to Brownian motions. The probability density distribution of finding a particle somewhere in space is Gaussian, which satisfies the classical Fickian diffusion equation \cite{MeeSik,MetKla04}.

Note that assumptions (i) and (ii) hold for diffusive transport of solute in homogeneous porous media, where solute plumes were observed to decay exponentially \cite{Bear61,Bear72} and so can be described accurately by integer-order diffusion equations. However, field tests showed that solute transport in heterogeneous aquifers often exhibit highly skewed and power-law decaying behavior, while sFDEs were derived under the assumption that the solutions have such behavior \cite{BenSchMee,MeeSik,MetKla04}. This is why sFDEs can accurately describe the solute transport in heterogeneous media more accurately than integer-order diffusion equations do. Consequently, they have attracted extensive research activities in the last few decades \cite{ErvRoo05,LiZha,LiChe,LiuAnh}.

However, sFDEs present new mathematical and numerical issues that are not common in the context of integer-order diffusion equations. Because of their nonlocal nature, numerical discretizations of sFDEs usually yield dense or full stiffness matrices \cite{Den,LiZha,LiuAnh,Roo}. A direct solver typically has $O(N^3)$ computational complexity and $O(N^2)$ memory requirement. A conventional Krylov subspace iterative method has $O(N^2)$ computational complexity per iteration, but may diverge due to significant amount of round-off errors \cite{WanDu13a}. In any case, the significantly increased computational complexity of numerical discretizations of sFDEs compared to their integer-order analogues is deemed computationally intractable for realistic simulations in multiple space dimensions, especially when parameter learning or control of the systems is involved.

It was discovered that the stiffness matrices of the numerical discretizations of constant-order sFDEs on a uniform partition typically possess a Toeplitz-like structure \cite{WanWanSir}, which reduces the memory requirement from $O(N^2)$ to $O(N)$ and computational complexity from $O(N^3)$ to $O(N\log N)$ per Krylov subspace iteration via the discrete fast Fourier transform. Furthermore, different preconditioners were employed to futher improve the computational efficiency and even convergence behavior \cite{BaiPan,Daniele,JinF,Jin3,LinNGSun,LinNGSun1,PanNgWang,WanDu13a,ZJinL}.

However, Ervin et al. \cite{ErvHeu16} proved that the one-dimensional constant-order constant-coefficient linear sFDEs with smooth right-hand side generate solutions with singularity at the end points of the spatial interval, which is in sharp contrast to their integer-order analogues and makes the error estimates of their numerical approximations derived under full regularity assumptions inappropriate. The singularity of the solutions to constant-order sFDEs seems to be physically irrelevant to the diffusive process of the solute, and occurs due to the incompatibility between the nonlocality of the power law decaying tails of the sFDEs inside the domain and the locality of the imposed classical boundary conditions. Intuitively, a physically relevant sFDE model should not only properly model the anomalous transport of solutes in heterogeneous porous media, but also correct the non-physical behavior of solutions to the existing sFDEs and thus maintain the smoothing nature of the diffusive transport process.

Recently, the wellposedness and smoothing properties of a variable-order linear sFDE was analyzed in \cite{ZheWan}: If the variable order has an integer limit at the boundary then the solutions have the full regularity as their integer-order analogues do; Otherwise, the solutions exhibit certain singularity at the boundary as their constant-order sFDE analogues do. Thus, the variable-order sFDE provides a feasible approach to resolve the non-physical singularity of solutions to  constant-order sFDEs near the boundary while retaining their advantages. In fact, variable-order sFDEs have been used in many applications \cite{SunChaZha,SunCheChe}, as the variable order is closely related to the fractal dimension of the porous media via the Hurst index \cite{EmbMae,MeeSik} and so can account for the changes of the geometrical structure or properties of the media. 

Some numerical studies of variable-order FDEs can be found in the literature in recent years. First-order convergence rates were proved for finite difference methods for space-time-dependent variable-order space-fractional advection-diffusion equations in one space dimension under (the artificially assumed) full regularity assumptions of the true solutions without addressing the singularity issue of the problem \cite{ZhuLiu}. A spectral collocation method using  weighted Jacobi polynomials was derived for variable-order sFDEs \cite{ZenZhaKar}, in which numerical experiments were presented to demonstrate the utility of the proposed method. The stability and convergence of an implicit alternating direct method was proved in \cite{CheLiuBur} under the assumptions of the smoothness of the true solutions and (somewhat artificial) monotonicity of the diffusivity coefficients. However, due to the impact of variable order of the FDEs, the numerical discretizations of variable-order sFDEs no longer have Toeplitz-like stiffness matrices, so the fast solvers developed for constant-order sFDEs do not apply. Also, the spectral methods do not have diagonal stiffness matrices.

In this paper we develop a fast numerical solution technique for an indirect collocation method to the homogeneous Dirichlet boundary-value problem of a one-sided variable-order sFDE in one space dimension. We approximate the stiffness matrix by a finite sum of Toeplitz-like matrices that is asymptotically convergent to the stiffness matrix. Then we develop a fast divided and conquer (DAC) solver for the approximated system by employing the Toeplitz-like structures of each summand to reduce the computational complexity from $O(N^2)$ to $O(kN\log^3N)$ and the memory requirement from $O(N^2)$ to $O(kN\log^2 N)$.

The rest of the paper is organized as follows. In Section 2 we present the model problem and its numerical discretization. In Section 3 we approximate the coefficient matrix by a sum of Toeplitz-like matrices and analyze its asymptotic consistency. In Section 4 we develop a fast DAC method for the approximated system. We perform numerical experiments to test the performance of the method in the last section.

\section{A variable-order sFDE model and its indirect collocation method}
\subsection{Model problem}

An sFDE of order $1<\alpha<2$ was proposed in \cite{del} to model the anomalously superdiffusive transport of solute in heterogeneous porous media
\begin{equation}\label{sFDE2}\begin{array}{c}
-u^{\prime \prime}(x) - d\,{}_0^C D_x^{\alpha}u(x)  = f(x),~~x\in [0,1];\\[0.05in]
u(0)=u(1)=0,
\end{array}\end{equation}
where $u^{\prime \prime}$ refers to the second-order derivative of $u$ and the Caputo fractional derivative ${}_0^C D_x^{\alpha} g$ is defined by \cite{Pod}
$$\ds {}_0^C D_x^\alpha g(x) := \frac{1}{\Gamma(2-\alpha)} \int_0^x \frac{g''(s)}{(x-s)^{\alpha-1}}ds.$$

Equation \eqref{sFDE2} is proposed based on the fact that a large amount of solute particles may travel through high permeability zones in a superdiffusive manner \cite{BenSchMee,MeeSik}, which may deviate from the transport of the solute particles in the bulk fluid phase that undergo a Fickian diffusive transport \cite{Bear72}. Therefore, in model (\ref{sFDE2}), the $-u^{\prime \prime}$ term represents the $1/(1+d)$ portion of the total solute mass undergoing the Fickian diffusive while the $-d\;{}_0^C D_x^{\alpha} u$ term refers to the $d/(1+d)$ portion of the total solute mass undergoing the superdiffusive transport in high permeability zones.

Note that in realistic applications, the reservoir may consist of different types of porous media that have different fractional dimensions. Hence, in this paper we consider the homogeneous Dirichlet boundary-value problem of the following one-sided variable-order linear sFDE as a variable extension of \eqref{sFDE2}
\begin{equation}\label{s1:e1}\begin{array}{c}
-u^{\prime \prime}(x) - d(x) ~{}_0^C D_x^{\alpha(x)} u(x) = f(x), ~~x \in (0,1); \\[0.1in]
u(0) = u(1) = 0,
\end{array}\end{equation}
where $1\leq\alpha_{min}\leq \alpha(x)\leq\alpha_{max}<2$ ($\alpha(x)\not\equiv \alpha_{min}$ if $\alpha_{min}=1$) and $d(x)\geq 0$ is the fractional diffusivity. The variable-order Caputo fractional derivative
${}_0D_x^{\alpha(x)}g$ is defined by \cite{ZenZhaKar,ZhuLiu}
$$\ds {}_0D_x^{\alpha(x)}g(x):=\frac{1}{\Gamma(2-\alpha(x))}\int_0^x\frac{g''(s)}{(x-s)^{\alpha(x)-1}}ds$$
where $\Gamma(\cdot)$ refers to the Gamma function.

\subsection{An indirect collocation method}

We rewrite (\ref{s1:e1}) in terms of $v(x) := u^{\prime \prime}(x)$
\begin{equation}\label{inteq}
v(x)+\frac{d(x)}{\Gamma(2-\alpha(x))}
\int_0^x\frac{v(s)}{(x-s)^{\alpha(x)-1}}ds=-f(x).
\end{equation}
Then the solution $u$ to model (\ref{s1:e1}) can be obtained by postprocessing
\begin{equation}\label{inteq2}
u(x)=\int_0^xv(s)(x-s)ds-x\int_0^1v(s)(1-s)ds.
\end{equation}

Let $0=x_0<x_1<\cdots<x_{N+1}=1$ be a uniform partition of $[0,1]$ with $x_n=nh$ for $n=0,1,\cdots,N+1$ and $h=1/(N+1)$. Let $\phi_n(x)$ be the piecewise-linear basis functions with $\phi_n(x_n)=1$ and $\phi_n(x_m)=0$ for $m\neq n$. Each element $v_h(x)$ in the space $S_h$ of continuous and piecewise linear functions on $x\in [0,1]$ can be represented by
\begin{equation}\label{s1:e3}
v_h(x)=\sum_{n=0}^{N+1}v_n\phi_n(x),~~v_n:=v_h(x_n).
\end{equation}
 Then an indirect collocation method for model (\ref{s1:e1}) reads:
\paragraph{Step 1} Find $v_h(x)\in S_h$ such that for $0\le n\le N+1$
\begin{equation}\label{s1:e4}\begin{array}{l}
\ds v_h(x_n)+\frac{d(x_n)}{\Gamma(2-\alpha(x_n))}\int_0^{x_n}
\frac{v_h(s)}{(x_n-s)^{\alpha(x_n)-1}}ds=-f(x_n);
\end{array}
\end{equation}
\paragraph{Step 2} Obtain an approximation $u_h(x)$ of $u(x)$ by
\begin{equation}\label{s1:e5}\begin{array}{l}
\ds u_h(x):=\int_0^xv_h(s)(x-s)ds-x\int_0^1v_h(s)(1-s)ds,\quad x\in [0,1].
\end{array}
\end{equation}
The following error estimate of the method was proved \cite{ZheWan} under the assumptions that $d$, $f$ and $\alpha$ are second-order continuously differentiable on $[0,1]$, and $1 \leq \alpha_{min}\leq\alpha(x)\leq \alpha_{max}<2$ with $\alpha(x) \not \equiv \alpha_{min}$ if $\alpha_{min} = 1$
\begin{equation}\label{n1:e1}\begin{array}{l}
\ds \|u-u_h\|:=\max_{1\leq n\leq N}|u(x_n)-u_h(x_n)|\leq \left\{\begin{array}{ll}
\ds QN^{-(3-\alpha_{max})},&~~\alpha(0)>1,\\[0.075in]
\ds QN^{-2},&~~\alpha(0)=1.
\end{array}
\right.
\end{array}\end{equation}
Instead, for $u$ sufficiently smooth, we have $\|u-u_h\|\leq QN^{-2}$.

\vspace{-0.1in}
\subsection{Solution of the numerical scheme}

For $n=0$ we obtain from (\ref{s1:e4}) that $v_0=-f(0)$. We then plug (\ref{s1:e3}) into (\ref{s1:e4}) for $1\leq n\leq N+1$ and move the terms containing $v_0$ to the right-hand side of the resulting equation to get a linear system with the unknowns  $\{v_n\}_{n=1}^{N+1}$
\begin{equation}\label{s1:e6}\begin{array}{l}
\ds \mathbf A\mathbf v=\mbf f,
\end{array}
\end{equation}
where $\mathbf A=(A_{i,j})_{i,j=1}^{N+1}\in \mathbb R^{(N+1)\times (N+1)}$ is a lower triangular matrix of the form
\begin{equation}\label{s1:e8}\begin{array}{l}
\ds A_{i,j}=\left\{\begin{array}{ll}
\ds 1+\frac{d(x_{i})h^{2-\alpha(x_i)}}{\Gamma(4-\alpha(x_i))},
& i=j,\\[0.15in]
\ds \frac{d(x_{i})h^{2-\alpha(x_i)}}{\Gamma(4-\alpha(x_i))}t_{i,j},& 1\le j\le i-1,\\[0.15in]
\ds 0,&\mbox{otherwise},
\end{array}\right.
\end{array}\end{equation}
\begin{equation}\label{s1:e8a}
t_{i,j} = (i-j-1)^{3-\alpha(x_i)}-2(i-j)^{3-\alpha(x_i)}+(i-j+1)^{3-\alpha(x_i)},
\end{equation}
$\mathbf v:=[v_1,v_2,\cdots,v_{N+1}]^T$ and $\mbf f:=[f_1,f_2,\cdots,f_{N+1}]^T$ with
$f_n$ given by
$$\begin{array}{l}
\ds f_n=-f(x_n)-d(x_n)v_0\Big(\frac{x_n^{2-\alpha(x_n)}}{\Gamma(3-\alpha(x_n))}
+\frac{(x_n-x_1)^{3-\alpha(x_n)}-x_n^{3-\alpha(x_n)}}{h\Gamma(4-\alpha(x_n))}\Big).
\end{array}$$
With $v_h(x)$ obtained from (\ref{s1:e6}), we plug $v_h(x)$ into (\ref{s1:e5}) to obtain the approximation $u_h(x)$ to $u(x)$. In particular, the discrete approximation $\mathbf u=[u_1,u_2,\cdots,u_N]^T$ with $u_n:=u_h(x_n)$ can be evaluated inductively as follows.
\begin{theorem}\label{thmind}
The $\{u_n\}_{n=1}^N$ can be computed by
\begin{equation}\label{thmind:e1}
\begin{array}{l}
\ds u_1=\frac{h^2(2v_0+v_{1})}{6}-hI,\\[0.1in]
\ds u_{n+1}=u_n+hV_{n}+\frac{h^2(2v_n+v_{n+1})}{6}-hI,~~1\leq n\leq N-1,
\end{array}
\end{equation}
where $I=\int_0^1v_h(s)(1-s)ds$
and $\{V_n\}_{n=0}^{N-1}$ is inductively generated by
\begin{equation}\label{V}\begin{array}{ll}
\ds V_0=\frac{h(v_0+v_1)}{2},&\ds V_{n}=V_{n-1}+\frac{h(v_{n}+v_{n+1})}{2},~~1\leq n\leq N-1.
\end{array}\end{equation}
\end{theorem}
\begin{proof}
For (\ref{s1:e5}) with $x=x_1$, we obtain the first equation of (\ref{thmind:e1}) by a direct calculation. For (\ref{s1:e5}) with $x=x_{n+1}$ and $1\leq n\leq N-1$, we split the first integral on $(0,x_{n+1})$ to those on $(0,x_n)$ and $(x_n,x_{n+1})$ and split the kernel $(x_{n+1}-s)$ to $(x_{n}-s)+h$ in the first resulting integral to obtain
\begin{equation}\label{thmind:e2}\begin{array}{rl}
\ds u_{n+1}&\ds=\int_0^{x_{n+1}}v_h(s)(x_{n+1}-s)ds-x_{n+1}\int_0^1v_h(s)(1-s)ds\\[0.15in]
&\ds =\int_0^{x_{n}}v_h(s)(x_{n+1}-s)ds
+\int_{x_n}^{x_{n+1}}v_h(s)(x_{n+1}-s)ds-x_{n+1}I\\[0.15in]
&\ds=\int_0^{x_{n}}v_h(s)(x_{n}-s)ds+h\int_0^{x_{n}}v_h(s)ds\\[0.15in]
&\ds\quad\quad +\frac{h^2(2v_n+v_{n+1})}{6}-x_{n+1}I.
\end{array}\end{equation}

Note that as $v_h(x)$ is a piecewise linear function, $V_n$ defined by (\ref{V}) represents the integration of $v_h(x)$ over $(0,x_{n})$. Therefore, the second term on the right-hand side of (\ref{thmind:e2}) equals to $hV_n$. From (\ref{s1:e5}) with $x=x_{n}$ we have
$$\begin{array}{rl}
\ds u_{n}&\ds=\int_0^{x_{n}}v_h(s)(x_{n}-s)ds-x_{n}I.
\end{array}$$
We plug this into (\ref{thmind:e2}) to get
$$\begin{array}{ll}
\ds u_{n+1}&\ds=u_n+x_nI+hV_{n}+\frac{h^2(2v_n+v_{n+1})}{6}-x_{n+1}I\\[0.125in]
\ds&\ds=u_n+hV_{n}+\frac{h^2(2v_n+v_{n+1})}{6}-hI,
\end{array}$$
which finishes the proof.
\end{proof}

\section{An approximated collocation scheme}

As $\mathbf A$ is a lower triangular matrix, a direct solution of the numerical scheme requires $O(N^2)$ storage and has $O(N^2)$ computational complexity. Once we obtained $\mbf v$, only $O(N)$ storage and $O(N)$ operations are needed to compute $\mbf u$ by Theorem \ref{thmind}. That is, the main task lies in reducing the memory requirement and the computational complexity of (\ref{s1:e6}). To do so, we approximate $\mbf A$ by a finite sum of Toeplitz-like matrices.

\begin{theorem}\label{thm2}
For $1\leq j\leq i-1$, $t_{i,j}$ defined by (\ref{s1:e8a}) can be approximated by
\begin{equation}\label{thm2:e1}\begin{array}{l}
\ds t_{i,j}\approx 2(i-j)^{3-\bar{\alpha}}\Big[1+(\bar{\alpha}-\alpha(x_i))\ln(i-j)+
\frac{(\bar{\alpha}-\alpha(x_i))^2}{2!}\ln^2(i-j)\\[0.15in]
\qquad\quad\ds+\cdots+
\frac{\big(\bar{\alpha}-\alpha(x_i)\big)^s}{s!}\ln^{s}(i-j)\Big]\cdot \bigg[{{3-\alpha(x_i)}\choose {2}}\frac{1}{(i-j)^2}\\[0.15in]
\ds\qquad\quad +{{3-\alpha(x_i)}\choose {4}}\frac{1}{(i-j)^4}+\cdots+{3-\alpha(x_i)\choose 2k}\frac{1}{(i-j)^{2k}}\bigg]
\end{array}\end{equation}
for some $k,s\in \mathbb N^+$ with the residue $R^{i,j}_{s,k}$ given by
\begin{equation}\label{sub1:e10}\begin{array}{l}
\ds\hspace{-0.1in} R^{i,j}_{s,k}=(i-j)^{3-\bar{\alpha}}\hat R^{i,j}_{s+1}
\bigg [ \Big(1-\frac{1}{i-j}\Big)^{3-\alpha(x_i)} -2+\Big(1+\frac{1}{i-j}\Big)^{3-\alpha(x_i)}\,\bigg]\\[0.175in]
\ds\qquad\qquad\qquad+(i-j)^{3-\alpha(x_i)}R^{i,j}_{2k+2}
-(i-j)^{3-\bar{\alpha}}\hat R^{i,j}_{s+1}R^{i,j}_{2k+2}.
\end{array}
\end{equation}
\end{theorem}

\begin{proof} We decouple the nonlinear dependence of $i$ and $j$ in $t_{i,j}$ (cf (\ref{s1:e8a})) for $j\leq i-1$ by applying the Taylor's expansion of power functions to the first and the third terms on the right-hand side of (\ref{s1:e8a}) as follows
\begin{equation}\label{sub1:e3}\begin{array}{l}
\ds (i-j \pm 1)^{3-\alpha(x_i)} = (i-j)^{3-\alpha(x_i)}\Big(1 \pm \frac{1}{i-j}\Big)^{3-\alpha(x_i)}\\[0.15in]
\ds \quad = (i-j)^{3-\alpha(x_i)}\bigg [1 \pm {{3-\alpha(x_i)}\choose{1}}\frac{1}{i-j}
+{{3-\alpha(x_i)}\choose{2}}\frac{1}{(i-j)^2}\\[0.15in]
\ds \qquad \pm {{3-\alpha(x_i)}\choose{3}}\frac{1}{(i-j)^3} +\cdots+{3-\alpha(x_i)\choose 2k}\frac{1}{{(i-j)}^{2k}}
\\[0.15in]
\ds\qquad \pm {3-\alpha(x_i)\choose 2k+1}\frac{1}{{(i-j)}^{2k+1}} + R_{2k+2}^{i,j,\pm} \bigg],
\end{array}\end{equation}
where $R_{2k+2}^{i,j,pm}$ are the remainders of the Taylor's expansion
\begin{equation}\label{sub:e4}
R_{2k+2}^{i,j,\pm} :={3-\alpha(x_i)\choose 2k+2}\Big(1 \pm \frac{\theta_\pm}{i-j}\Big)^{1-\alpha(x_i)-2k}
\frac1{(i-j)^{2k+2}},~~\theta_\pm \in (0,1).
\end{equation}
Here $\theta_\pm = \theta_\pm(i,j,k)$. We plug (\ref{sub1:e3}) into (\ref{s1:e8a}) to obtain
\begin{equation}\label{sub1:e4}\begin{array}{l}
\ds t_{i,j} = (i-j)^{3-\alpha(x_i)} \bigg [ {{3-\alpha(x_i)}\choose {2}}\frac{2}{(i-j)^2}
+{{3-\alpha(x_i)}\choose {4}}\frac{2}{(i-j)^4}+\cdots\\[0.15in]
\ds\qquad\qquad\qquad+{3-\alpha(x_i)\choose 2k}\frac{2}{(i-j)^{2k}} + R^{i,j}_{2k+2} \bigg],\\[-0.25in]
\end{array}\end{equation}
\begin{equation}\label{s1:e8b}\begin{array}{ll}
\ds R^{i,j}_{2k+2}&:\ds=R_{2k+2}^{i,j,-}+R_{2k+2}^{i,j,+}={3-\alpha(x_i)\choose 2k+2}\frac{1}{(i-j)^{2k+2}}\\[0.2in]
\ds &\ds\qquad\times\bigg [\Big(1-\frac{\theta_-}{i-j}\Big)^{1-\alpha(x_i)-2k}
+\Big(1+\frac{\theta_+}{i-j}\Big)^{1-\alpha(x_i)-2k}\,\bigg].
\end{array}\end{equation}

Due to the impact of the variable order $\alpha(x)$, the matrix of entries $(i-j)^{3-\alpha(x_i)}$ is not Toeplitz. The assumption on $\alpha(x)$ (preceding \eqref{n1:e1}) implies $\bar{\alpha}:=(\alpha_{max}+\alpha_{min})/2 > 1$.  We use the Taylor's expansion of $a^x$ with $a>0$ and $0<x<1$ to obtain
\begin{equation}\label{sub1:e66}\begin{array}{rl}
\ds (i-j)^{3-\alpha(x_i)}&\ds=(i-j)^{\bar{\alpha}-\alpha(x_i)}(i-j)^{3-\bar{\alpha}}\\[0.025in]
&\ds = \bigg[1 + \big(\bar{\alpha}-\alpha(x_i)\big)\ln(i-j) 
+ \frac{\big(\bar{\alpha}-\alpha(x_i)\big)^2}{2!}\ln^2(i-j) \\[0.1in]
& \ds \quad + \cdots +\frac{(\bar{\alpha}-\alpha(x_i))^s}{s!}\ln^s(i-j)+\hat R^{i,j}_{s+1}\bigg](i-j)^{3-\bar{\alpha}},
\end{array}\end{equation}
\begin{equation}\label{sub1:e6}\begin{array}{l}
\ds \hat R^{i,j}_{s+1}=\frac{(\bar{\alpha}-\eta)^{s+1} \ln^{s+1} (i-j)}{(s+1)!},\quad \eta\in
\left\{\begin{array}{l}
\ds (\bar{\alpha},\alpha(x_i)),~~\mbox{if}~\bar\alpha\leq \alpha(x_i),\\[0.1in]
(\alpha(x_i),\bar{\alpha}),~~\mbox{otherwise}.
\end{array}
\right.
\end{array}\end{equation}

Substituting (\ref{sub1:e66}) into (\ref{sub1:e4}) yields
\begin{equation}\label{sub1:e8}\begin{array}{l}
\ds t_{i,j}= 2(i-j)^{3-\bar{\alpha}}\bigg[1+\big(\bar{\alpha}-\alpha(x_i)\big)\ln(i-j)+
\frac{\big(\bar{\alpha}-\alpha(x_i)\big)^2}{2!}\ln^2(i-j)\\[0.15in]
\qquad\quad\ds+\cdots+
\frac{\big(\bar{\alpha}-\alpha(x_i)\big)^s}{s!}\ln^{s}(i-j)+\hat R_{s+1}^{i,j}\bigg] \bigg[{{3-\alpha(x_i)}\choose {2}}\frac{1}{(i-j)^2}\\[0.15in]
\ds\qquad\quad +{{3-\alpha(x_i)}\choose {4}}\frac{1}{(i-j)^4}+\cdots+{3-\alpha(x_i)\choose 2k}\frac{1}{(i-j)^{2k}}+R^{i,j}_{2k+2}\bigg]\\[0.15in]
\ds\quad\,\,\,= 2(i-j)^{3-\bar{\alpha}}\bigg [ 1+(\bar{\alpha}-\alpha(x_i))\ln(i-j)+
\frac{(\bar{\alpha}-\alpha(x_i))^2}{2!}\ln^2(i-j)\\[0.15in]
\qquad\quad\ds+\cdots+ \frac{\big(\bar{\alpha}-\alpha(x_i)\big)^s}{s!}\ln^{s}(i-j) \bigg] \bigg[{{3-\alpha(x_i)}\choose {2}}\frac{1}{(i-j)^2}\\[0.15in]
\ds\qquad\quad +{{3-\alpha(x_i)}\choose {4}}\frac{1}{(i-j)^4}+\cdots+{3-\alpha(x_i)\choose 2k}\frac{1}{(i-j)^{2k}}\bigg]+R^{i,j}_{s,k},
\end{array}\end{equation}
$$\begin{array}{l}
\ds R_{s,k}^{i,j}  = 2(i-j)^{3-\bar \alpha} \bigg[ \hat R^{i,j}_{s+1} \bigg(R^{i,j}_{2k+2} + \sum_{m=0}^k {{3-\alpha(x_i)}\choose {2m}}\frac{1}{(i-j)^{2m}} \bigg)\\[0.15in]
\ds \qquad\qquad +R^{i,j}_{2k+2} \bigg (\sum_{m=0}^s\frac{\big(\bar{\alpha}-\alpha(x_i)\big)^m}{m!}\ln^m(i-j) +\hat R_{s+1} \bigg) -R^{i,j}_{2k+2}\hat R^{i,j}_{s+1}\bigg].
\end{array}$$
Then we replace the terms in the first and second brackets by (\ref{sub1:e4}) and (\ref{sub1:e66}), respectively, to obtain (\ref{sub1:e10}). Dropping the truncation error $R^{i,j}_{s,k}$ in (\ref{sub1:e8}), we get the approximation (\ref{thm2:e1}) of $t_{i,j}$ for $j\leq i-1$.
\end{proof}

Replacing $t_{i,j}$ in the entries of $\mbf A$ by the right-hand side of (\ref{thm2:e1}) leads to the corresponding approximated system 
\begin{equation}\label{app}
\tilde{\mbf A} {\mbf v}=\mbf f,~~\tilde{\mbf A}:=(\tilde A_{i,j})_{i,j=1}^{N+1}.
\end{equation}

\begin{theorem}\label{thm3}
For $j\leq i-1$ and $k,s\in \mathbb N^+$ with $ s>e\ln N/2-1$, the local truncation error $R^{i,j}_{s,k}$ can be bounded by
\begin{equation}\label{thm3:e1}\begin{array}{l}
\ds |R_{s,k}^{i,j}| \leq\frac{Q}{\sqrt{s+1}{(i-j)^{\bar{\alpha}-1}}}\\[0.15in]
\quad\qquad\ds\quad
+\left\{\begin{array}{ll}
\ds \frac1{(2k)^{4-\alpha(x_i)}},&~i-j=1;\\[0.05in]
\ds \frac{1}{(2k)^{4-\alpha(x_i)}(i-j-1)^{2k+\alpha(x_i)-1}},&~i-j\geq 2.
\end{array}
\right.
\end{array}
\end{equation}
Consequently, we have for $i-j\geq 2$
\begin{equation}\label{thm3:e2}\begin{array}{ll}
\ds\max_{1\le i,j\le N+1}|\tilde{A}_{i,j}-A_{i,j}|&\ds\le Ch^{2-\alpha_{max}}\Big(\frac{1}{\sqrt{s+1}(i-j)^{\bar{\alpha}-1}}\\
&\ds\quad
+\frac{1}{(2k)^{4-\alpha_{max}}(i-j-1)^{2k+\alpha_{max}-1}} \Big).
\end{array}
\end{equation}

That is, the approximation scheme (\ref{app}) is asymptotically consistent with the original problem (\ref{s1:e6}) with respect to $s$ and $k$.
\end{theorem}

\begin{proof} We bound $R^{i,j}_{2k+2}$ and $\hat R^{i,j}_{s+1}$ respectively. The binomial coefficients in (\ref{s1:e8b}) can be represented in terms of the gamma functions as follows 
\begin{equation*}\begin{array}{ll}
\ds {3-\alpha(x_i)\choose 2k+2}
&\ds=\frac{(3-\alpha(x_i))(3-\alpha(x_i)-1)\cdots(3-\alpha(x_i)-(2k+1))}{(2k+2)!}\\[0.15in]
&\hspace{-0.55in}\ds=\frac{(-1)^{2k+2}(2k+1-(3-\alpha(x_i)))(2k-(3-\alpha(x_i)))\cdots(0-(3-\alpha(x_i)))}{(2k+2)!}\\[0.15in]
&\hspace{-0.55in}\ds=(-1)^{2k+2}\frac{\Gamma(2k+2-(3-\alpha(x_i)))}{\Gamma(2k+3)\Gamma(2-(3-\alpha(x_i)))}(1-(3-\alpha(x_i)))(0-(3-\alpha(x_i)))\\[0.15in]
&\hspace{-0.55in}\ds=(-1)^{2k+2}(\alpha(x_i)-2)(\alpha(x_i)-3)\frac{\Gamma(2k+\alpha(x_i)-1)}{\Gamma(2k+3)\Gamma(\alpha(x_i)-1)}.
\end{array}
\end{equation*}
We use the asymptotic expansions of Gamma functions (\cite[Eq. 1.5.15]{KilSri})
$$\begin{array}{l}
\ds \frac{\Gamma(z+a)}{\Gamma(z+b)}=z^{a-b}\Big(1+O\Big(\frac{1}{z}\Big)\Big),~~z+a>0,~~z\rightarrow +\infty, \\[-0.1in]
\end{array}$$
to get
\begin{equation}\label{sub1:e6.5}\begin{array}{ll}
\ds \Big|{3-\alpha(x_i)\choose 2k+2}\Big|\le \frac{Q}{(2k)^{4-\alpha(x_i)}}.
\end{array}
\end{equation}
We note that for $i-j=1$ the left-hand side of (\ref{sub1:e3}) with the minus sign in ``$\pm$'' vanishes, so $R^{i,j,-}_{2k+2}$, and thus the first term on the right-hand side of the equal sign``=" in (\ref{sub1:e7}), vanishes.  We thus bound the remaining factors on the right hand of (\ref{s1:e8b}) by
\begin{equation}\label{sub1:e7}\begin{array}{ll}
\ds \bigg| \bigg( \Big (1 - \frac{\theta_-}{i-j} \Big)^{1-\alpha(x_i)-2k}
+ \Big ( 1 +\frac{\theta_+}{i-j} \Big)^{1-\alpha(x_i)-2k}\Big)\frac{1}{(i-j)^{2k+2}} \bigg|\\[0.15in]
\ds~=\frac{1}{(i-j)^{3-\alpha(x_i)}}\bigg(\frac{1}{(i-j-{\theta_-})^{2k+\alpha(x_i)-1}} +
\frac{1}{(i-j+{\theta_+})^{2k+\alpha(x_i)-1}} \bigg)\\[0.175in]
\ds~\le \left\{\begin{array}{ll}
\ds 1,&~~i-j=1;\\[0.1in]
\ds \frac{2(i-j)^{\alpha(x_i)-3}}{(i-j-\theta_-)^{2k+\alpha(x_i)-1}},&~~i-j\geq 2,~~\theta_-\in (0,1),
\end{array}
\right.
\end{array}
\end{equation}
We incorporate the proceeding estimates to obtain the estimates of the local truncation error $R_{2k+2}^{i,j}$
\begin{equation}\label{aa}\begin{array}{l}
\ds \hspace{-0.15in}|R_{2k+2}^{i,j}|\le  \left\{
\begin{array}{ll}
\ds \frac{Q}{(2k)^{4-\alpha(x_i)}},&~~i-j=1;\\[0.15in]
\ds \frac{Q(i-j)^{\alpha(x_i)-3}}{(2k)^{4-\alpha(x_i)}(i-j-\theta_-)^{2k+\alpha(x_i)-1}},&~i-j\geq 2,~~\theta_-\in (0,1).
\end{array}
\right.
\end{array}
\end{equation}


We note that $\eta$ in (\ref{sub1:e6}) satisfies $(\bar{\alpha}-\eta)^{s+1}\le 2^{-(s+1)}$ and use the Stirling's formula for gamma functions to get 
\begin{equation*}\begin{array}{l}
\ds \Big|\frac{\ln^{s+1} (i-j)}{(s+1)!}\Big| \leq\frac{Q\ln^{s+1}(i-j)}{\sqrt{2\pi(s+1)}
\big((s+1)/e\big)^{s+1}} \le \frac{Q}{\sqrt{s+1}}\Big(\frac{e\ln (i-j)}{s+1}\Big)^{s+1}.
\end{array}\end{equation*}
We thus bound $\hat R_{s+1}^{i,j}$ by
\begin{equation}\label{sub1:e8b}\begin{array}{l}
\ds |\hat R_{s+1}^{i,j}|\le \frac{Q}{\sqrt{s+1}}
\Big(\frac{e\ln (i-j)}{2(s+1)}\Big)^{s+1}.
\end{array}
\end{equation}
We combine (\ref{aa}) and (\ref{sub1:e8b}) to bound $R^{i,j}_{s,k}$ in (\ref{sub1:e10}) by
\begin{equation}\label{sub1:e9}\begin{array}{ll}
\ds \hspace{-0.15in} |R_{s,k}^{i,j}|&\ds\le (i-j)^{3-\bar{\alpha}} \big | \hat R_{s+1}^{i,j} \big |\,\bigg |\Big(1-\frac{1}{i-j}\Big)^{3-\alpha(x_i)} -2+\Big(1+\frac{1}{i-j}\Big)^{3-\alpha(x_i)} \bigg|\\[0.15in]
&\ds\qquad\qquad+(i-j)^{3-\alpha(x_i)} \big | R_{2k+2}^{i,j} \big | +
\big | R_{2k+2}^{i,j} \hat R_{s+1}^{i,j} \big |\\[0.1in]
&\ds\leq \frac{Q|\hat R_{s+1}^{i,j}|}{(i-j)^{\bar{\alpha}-1}}+
(i-j)^{3-\alpha(x_i)}|R_{2k+2}^{i,j}|
+|R_{2k+2}^{i,j}R_{s+1}^{i,j}|\\[0.15in]
&\ds\leq\frac{Q}{\sqrt{s+1}{(i-j)^{\bar{\alpha}-1}}}\Big(\frac{e\ln (i-j)}{2(s+1)}\Big)^{s+1}\\[0.15in]
&\ds\quad+\left\{
\begin{array}{ll}
\ds \frac{Q}{(2k)^{4-\alpha(x_i)}},&~i-j=1;\\[0.1in]
\ds \frac{Q}{(2k)^{4-\alpha(x_i)}(i-j-\theta_-)^{2k+\alpha(x_i)-1}},&~i-j\geq 2,~\theta_- \in (0,1),
\end{array}
\right.
\end{array}
\end{equation}
where in the first inequality we have used the Taylor's expansion as in (\ref{sub1:e3}) and the expression (\ref{sub:e4}) for $R^{i,j,\pm}_{2k+2}$ with $k=0$
$$\begin{array}{l}
\ds \Big|\Big(1-\frac{1}{i-j}\Big)^{3-\alpha(x_i)}
-2+\Big(1+\frac{1}{i-j}\Big)^{3-\alpha(x_i)}\Big|\\[0.15in]\ds \qquad\qquad= \left\{
\begin{array}{ll}
\ds 2^{3-\alpha(x_i)}-2\leq 2,&~i-j=1;\\[0.1in]
\ds |R^{i,j,1}_{2}+R^{i,j,2}_{2}|\leq\frac{Q}{(i-j)^2},&~i-j\geq 2.
\end{array}
\right.
\end{array} $$
By setting $\ds s>\frac{e\ln N}{2}-1$ in the first term on the right-hand side of (\ref{sub1:e9}) we obtain (\ref{thm3:e1}) and \eqref{thm3:e2}.
\end{proof}

\section{A fast divided-and-conquer solver}\label{DAC}

We develop a fast DAC solver for the collocation system (\ref{app}). Due to the impact of the variable order in the sFDE (\ref{s1:e1}), the stiffness matrix $\mathbf A$ in (\ref{s1:e6}) no longer has Toeplitz structure so the DAC algorithm \cite{MSun,FMWang} developed for constant-order FDEs cannot apply. We develop a fast DAC method by expressing the stiffness matrix $\mbf {\tilde A}$ in (\ref{app}) as
\begin{equation}\label{DAC:e4}\begin{array}{cc}
\mbf{\tilde A}&\ds=\left[\begin{array}{cc}
 \ds \mbf{\tilde A}_{N'}&\mbf 0\\[0.05in]
  \ds\mbf{\tilde{\Gamma}}_{N'}&\mbf{\hat{\tilde {A}}}_{N'}
\end{array}\right].
\end{array}
\end{equation}
\begin{algorithm}
\caption{The fast approximated DAC algorithm(denote $c=\lfloor \frac{e}{2}\ln N\rfloor $)}
\label{alg:B}
$$\begin{array}{l}
\ds\mbox{function} \mathbf v_N=\mbox{FDAC}(\mathbf
{\tilde{A}}_N,\mathbf f_N)\\[0.1in]
\ds \quad\mbox{if } N\le c\\[0.1in]
\ds \qquad ~\mathbf v_N=\mathbf {\tilde{A}}_N^{-1} \mathbf f_N\\[0.1in]
\ds \quad\mbox{else}\\[0.1in]
\ds\qquad \mathbf v_{N'}=\mbox{FDAC}(
\mathbf{\tilde {A}}_{N'},\mathbf f_{N'})\\[0.1in]
\ds\qquad \hat{\mathbf f}_{N'}=\hat{\mathbf f}_{N'}-\mathbf{\tilde{\Gamma}}_{N'}\mathbf v_{N'}\\[0.1in]
\ds\qquad \hat{\mathbf v}_{N'}=\mbox{FDAC}(\mathbf{\hat{ \tilde{A}}}_{N'},\hat{\mathbf f}_{N'})  \\[0.1in]
\ds \quad {\mbox{end if}}\\[0.1in]
\ds\mbox{end function}
\end{array}$$
\end{algorithm}
By Theorem \ref{thm3} the truncation errors $R^{i,j}_{s,k}$ are not necessarily small when $i-j$ is small, so the approximations of the corresponding entries may lose accuracy. To ensure the accuracy of the approximations, we evaluate entries on the right-up bands of $\mathbf{\Gamma}_{N'}$ with width ${c=\lceil\log N\rceil}$ (i.e., colored entries in the matrix below) exactly as follows
\begin{equation}\label{add1}\begin{array}{l}
\ds\mathbf{\tilde\Gamma}_{N'}=\mbf{\hat{\Gamma}}_{N'}+\mbf{\Psi}_{N'},\\[0.15in]
\mbf{\hat\Gamma}_{N'}\ds=
\left[\begin{array}{ccccccc}
\ds \tilde\Gamma_{1,1}&\cdots&
\tilde\Gamma_{1,N'-c}& 0 &0&
\cdots&0\\[0.1in]
\ds \tilde\Gamma_{2,1}&\cdots&\tilde\Gamma_{2,N'-c}&\tilde\Gamma_{2,N'-c+1}
&0&
\cdots&0\\[0.1in]
\vdots&\vdots&\ddots&\ddots&\ddots&\ddots&\vdots\\[0.1in]
\ds \tilde\Gamma_{c,1}&\cdots&
\tilde\Gamma_{c,N'-c}&\tilde\Gamma_{c,N'-c+1}
&\tilde\Gamma_{c,N'-c+2}
&\cdots&0\\[0.1in]
\vdots&\vdots&\ddots&\ddots&\ddots&\ddots&\vdots\\[0.1in]
\ds \tilde\Gamma_{N',1}&\cdots&\tilde\Gamma_{N',N'-c}&\tilde\Gamma_{N',N'-c+1}
&\tilde\Gamma_{N',N'-c+2}&\cdots&\tilde\Gamma_{N',N'}\\[0.15in]
\end{array}\right],\\[1in]
\mbf{\Psi}_{N'}=\ds
\left[\begin{array}{ccccccc}
\ds 0&\cdots&
0&{\color{blue}\Gamma_{1,N'-c+1}}&
{\color{blue}\Gamma_{1,N'-c+2}}&
\cdots&{\color{blue}\Gamma_{1,N'}}\\[0.1in]
\ds 0&\cdots&0&0
&{\color{blue}\Gamma_{2,N'-c+2}}&
\cdots&{\color{blue}\Gamma_{2,N'}}\\[0.1in]
\vdots&\vdots&\ddots&\ddots&\ddots&\ddots&\vdots\\[0.1in]
\ds 0&\cdots&
0&0
&0
&\cdots&{\color{blue}\Gamma_{c,N'}}\\[0.1in]
\vdots&\vdots&\ddots&\ddots&\ddots&\ddots&\vdots\\[0.1in]
\ds 0&0&0
&0&\cdots&0\\[0.1in]
\end{array}\right],
\end{array}
\end{equation}

 Thus, the matrix $\mbf{\Psi}_{N'}$ has $\frac{c(c+1)}{2}$ ($c=\lfloor \log N\rfloor$) nonzero entries totally, which indicates the matrix-vector multiplication $\mbf{\Psi}_{N'}\mbf v_{N'}$ can be computed exactly in $\frac{1}{2}\log^2N$ of memory and $\frac{1}{2}\log^2N$ of computational work.

\begin{theorem}\label{FAST:DAC}
The sub-matrix $\tilde{\mathbf{\Gamma}}_{N'}$ in (\ref{DAC:e4}) can be expressed as a sum of Toeplitz matrices multipied by diagonal matrices, so the corresponding system can be stored in $O(kN\log^2 N)$ and can be solved in $O(kN\log^3 N)$ operations by the fast approximated DAC method (see Algorithm 2).
\end{theorem}

\begin{proof} We combine (\ref{add1}) with (\ref{thm2:e1}) to express $\mathbf {\hat\Gamma}_{N'}$ as 
$$\begin{array}{l}
\ds\mathbf{\hat\Gamma}_{N'}= \mbox{diag}(\mathbf{K}^{0,1})\mathbf T^{0,1}+
\cdots+\mbox{diag}(\mathbf{K}^{s,1})\mathbf T^{s,1}+\cdots+\mbox{diag}(\mathbf{K}^{s,k})\mathbf T^{s,k}
\end{array}$$
where $\mathbf K^{p,q}:=(K^{p,q}_{i})_{i=1}^{N'}$ and $\mathbf T^{p,q}$ for $0\le p\le s, 1\le q\le k$ are given by
$$\begin{array}{l}
\ds  K^{p,q}_{j-N'}=\frac{2d(x_j)h^{2-\alpha(x_j)}}{\Gamma(4-\alpha(x_j))}
\frac{(\bar{\alpha}-\alpha(x_j))^{p-1}}{(p-1)!}
{3-\alpha(x_j)\choose 2q},~ N'+1\le j\le N,
\end{array}$$
and
$$\begin{array}{l}
\ds \mathbf T^{p,q}=\mbox{Toeplitz}(\mathbf t_c^{p,q},\mathbf t_r^{p,q})
\end{array}$$
 with $\mathbf t_{c}^{p,q}:=( t_{ci}^{p,q})_{i=1}^{N'}$ and $\mathbf t_r^{p,q}:=( t_{ri}^{p,q})_{i=1}^{N'}$ the first column and the first row
of $\mathbf T^{p,q}$, respectively

$$\begin{array}{ll}
\ds  t_{ci}^{p,q}=\frac{\ln^{p-1} (N'+i-1)}{(N'+i-1)^{\bar{\alpha}+2q-3}},~ ~1\le i\le N',\\[0.15in]
\ds  t_{ri}^{p,q}=\left\{\begin{array}{ll}
\ds \frac{\ln^{p-1}(N'-i+1)}
{(N'-i+1)^{\bar{\alpha}+2q-3}},& 1\le i\le N'-c,\\[0.15in]
0,& N'-c+1\le i\le N'.
\end{array}
\right.
\end{array}$$

It is well known that $\mathbf T^{p,q} \mathbf v_{N'}$ can be performed in $O(N'\log N')$ operations via the discrete fast Fourier transform (FFT), which implies that $\mathbf{\hat\Gamma}_{N'}\mathbf v_{N'}$ can be evaluated in $O(ksN'\log N')$ operations. To store $\mathbf{\hat \Gamma}_{N'}$, we only need to store $\mathbf{t}_c^{p,q}$, 
$\mathbf{t}_r^{p,q}$, and $\mathbf{K}^{p,q}$ that require $O(ksN')$ memory. If we take $s=\lceil{\frac{e}{2}\ln N-1}\rceil=O(\log N)$ according to Theorem \ref{thm3},
 the computational cost and the memory requirement of $\mbf{\hat{\Gamma}}_{N'}\mbf v_{N'}$ become $O(kN'\log N\log N')$ and $O(kN'\log N)$, respectively. Then the total number of computations ${\tilde\Theta}_N$ and the storage $\tilde{\Phi}_N$ of the proposed fast method can be evaluated by
$$\begin{array}{ll}
\ds \tilde{\Theta}_N&\ds=O(kN'\log N\log N')+2\tilde{\Theta}_{N'}
\\[0.15in]
&\ds = O \big (kN'\log N\log N' \big ) + 2O \Big (k\frac{N'}{2}\log N\log \frac{N'}{2}\Big)+4\tilde{\Theta}_{\frac{N'}{2}}\\[0.15in]
&\ds=\cdots= O\Big(kN'\log^2N \Big ( 1+2\frac{1}{2} + \cdots
+2^{J-1}\frac{1}{2^{J-1}} \Big) \Big)\\[0.15in]
& \ds = O(kN\log^3N),\\[0.15in]
\tilde{\Phi}_N&\ds=O(kN'\log N)+2\tilde{\Phi}_{N'}= O(kN'\log N)+2O \Big ( k\frac{N'}{2}\log N \Big ) +4\tilde{\Phi}_{\frac{N'}{2}}\\[0.15in]
&\ds=\cdots= O\Big(kN'\log N \Big (1+2\frac{1}{2}+\cdots+2^{J-1}\frac{1}{2^{J-1}} \Big) \Big)\\[0.15in]
&\ds = O\big (kJN'\log N \big ) = O\big (kN\log^2N \big).
\end{array}$$

By (\ref{add1}), the storage and the computational cost of $\mbf{\Psi}_{N'}\mbf v_{N'}$ are both $\frac{1}{2}\log^2 N$. Thus the total computational cost $\epsilon_{\theta}$ and the storage $\epsilon_{\phi}$ are
$$\begin{array}{ll}
\ds \epsilon_{\theta}&\ds=\frac{1}{2}\log^2 N+2\frac{1}{2}\log^2 N+\cdots+2^{J-1}\frac{1}{2}\log^2N\\[0.15in]
&\ds=\frac{1}{2}\log^2N\Big(1+2+\cdots+2^{J-1}\Big)=\frac{2^J-1}{2}\log^2N\le \frac{N}{2}\log^2N,\\[0.15in]
\ds\epsilon_{\phi}&\ds=\frac{1}{2}\log^2N+\frac{1}{2}\log^2N+
\cdots+\frac{1}{2}\log^2N\le \frac{1}{2} \frac{N}{\lfloor\log N\rfloor}\log^2 N\le N\log N.
\end{array}$$

Then totally flops $\Theta_N$ and storage $\Phi_N$ of the fast approximated algorithm can be estimated by
$$\begin{array}{ll}
\ds \Theta_N&=\ds\tilde{\Theta}_N+\epsilon_{\theta}
=O(N\log^3N)+\frac{N}{2}\log^2N =O(kN\log^3N),\\[0.15in]
\ds\Phi_N&\ds=\tilde{\Phi}_N+\epsilon_{\phi}=O(kN\log^2 N),
\end{array}$$
which finishes the proof.
\end{proof}

\begin{remark}
It takes $O(N^2)$ computational works to generate the entries of the original coefficient matrix $\mbf  A$, while in the fast approximated DAC method only $O(ks\log N)$ operations are needed to generate all the entries of the approximated coefficient matrix $\tilde{\mbf A}$. 
\end{remark}
%

\section{Numerical experiments}\label{numer}

We perform numerical experiments to investigate the performance of the fast DAC method (FDAC) by comparing it with the Forward Substitution method (FS). The approximated system (\ref{app}) will be solved and the convergence rates of the indirect collocation method, the CPU times (in seconds) of generating the coefficient matrix ($CPU_M$) and of solving the lower triangular linear system ($CPU_S$) will be recorded. We set the parameter $k$ appeared in Theorem \ref{thm2} equal to $2$ throughout the experiments. 

\paragraph{Experiment 1.}\label{numer1}
We set $d(x)=1$ and
$$\begin{array}{l}
\ds \alpha(x)=(\alpha_0-\alpha_1)\big(1-x-\frac{\sin(2\pi(1-x))}{2\pi}\big)+\alpha_1
\end{array}$$
with $\alpha_0=1.2$, $\alpha_1=1.6$. The exact solution $u(x)=x^4(1-x)$ and the corresponding right hand term $f(x)$ is given by
$$\begin{array}{l}
\ds f(x)=-(12x^2-20x^3)-\Big(\frac{24}{\Gamma(5-\alpha(x))}x^{4-\alpha(x)}
-\frac{120}{\Gamma(6-\alpha(x))}x^{5-\alpha(x)}\Big).
\end{array}$$
The results are presented in Table 1 and Table 2, from which we notice that
\begin{itemize}
\item The
CPU time consumed by FS increases at about a quadruple rate between two consecutive numbers of collocation grids while the increment of the CPU time of the FDAC is almost liner; 
\item  It is more efficient in the FDAC to generate the entries of the coefficient matrix than that in the FS. For instance, when $N=2^{15}$, the FS takes 352 seconds to compute the entries of $\mbf A$ while the FDAC only requires $3.78$ seconds;
\item when $N\geq 2^{11}$, the FDAC is more efficient than FS for solving the linear systems. For instance, for the case of $N=2^{15}$, the $CPU_S$ of FS is $99$ seconds while in the FDAC it is $1.5$ seconds; 
\item  The FS is out of memory for $N\geq 2^{16}$ while the FDAC still works even for $N\geq 2^{19}$; 
\item  The FDAC has almost the same accuracy and convergence rates as FS for relatively small $N$ while for large $N$ the convergence rate is affected by the round-off errors.

\end{itemize}

From the observations mentioned above, the presented FDAC has shown strong potentials for efficiently and effectively solving the variable-order space-fractional diffusion equations by reducing the memory requirement and improving the efficiency of generating entries of the coefficient matrix and solving the linear systems. This implies that the proposed fast method is particularly suitable for the large-scale simulations. 
\begin{table}
\renewcommand{\tablename}
\caption{\centering{ Table 1} \protect \\ \qquad \qquad \qquad \qquad \qquad Errors of 
FS and FDAC for Experiment 1}
\\
 $$ \begin{array}{|c|c|c|c|c|c|}
 \hline
 &FS&&FDAC&\\
 \hline
N&\|u-u_h\|&Order&\|u-u_h\|&Order\\\hline
   2^8&  $4.36009e-06$& &  $4.25781e-06$ &\\
   2^9  &$1.08719e-06$ & 2.00&$1.05751e-06$ & 2.01\\
   2^{10}&$2.71315e-07$ & 2.00&$2.65646e-07$ &  1.99 \\
   2^{11} &$6.77461e-08$& 2.00 &$6.62834e-08$ &2.00\\
   2^{12}& $1.69224e-08$&  2.00 & $1.70719e-08$&  1.96\\
   2^{13} & $4.22815e-09$& 2.00 &$ 4.38408e-09$&1.96 \\\hline
%

  \end{array}$$
\end{table}

\begin{table}

\renewcommand{\tablename}
\caption{\centering{ Table 2} \protect \\ \qquad \qquad \qquad CPUs of 
FS and FDAC for Experiment 1}
\\
  \centering
 $$ \begin{array}{|c|c|c|c|c|c|}
 \hline
 &FS&FS&FDAC&FDAC\\
 \hline
N&CPU_{M}&CPU_{S}&CPU_{M}&CPU_{S}\\\hline
   2^8&0.027& 0.002& 0.037&  0.007\\
   2^9 & 0.084& 0.005&0.040& 0.012\\
  2^{10}&0.315 &	0.027 &0.087 & 0.031\\
 2^{11} &1.27 & 0.179&  0.185&  0.058 \\
   2^{12}& 5.03& 0.801&0.385 &  0.190\\
   2^{13} & 19.76& 3.56& 0.889& 0.297 \\
   2^{14}&85.22 &21.12 & 1.87&   0.642\\
   2^{15} &352.54&98.9 & 3.78 & 1.51  \\

   2^{16} &$-$&$-$& 7.97& 2.20\\
  2^{17}&$-$&$-$&17.26 & 5.16\\
   2^{18} &$-$&$-$&36.14 & 12.03\\
   2^{19} &$-$&$-$&75.53& 27.88
  
  \\\hline
  \end{array}$$
\end{table}

%
%
%

\paragraph{Experiment 2.}\label{numer2}
Let $d(x)=1$ and $\alpha(x)=(\alpha_1-\alpha_0)x+\alpha_0$. We choose a solution with a boundary layer at $x=0$
$$\begin{array}{l}
\ds u(x)=\int_0^x s^{2-\alpha(s)}(x-s)ds-x\int_0^1s^{2-\alpha(s)}(1-s)ds.
\end{array}$$
The corresponding right-hand side is
$$\begin{array}{l}
\ds f(x)=-x^{2-\alpha(x)}-\frac{d(x)}{\Gamma(2-\alpha(x))}\int_0^xs^{2-\alpha(s)}(x-s)^{1-\alpha(x)}ds.
\end{array}$$

Numerical results are presented in Table 3--5, which is strongly consistent with the observations in Experiment 1.

\begin{table}
\renewcommand{\tablename}
\caption{\centering{ Table 3} \protect \\ \qquad \qquad \qquad Errors of FS and FDAC for Experiment 2  with $\alpha_0=1.2$ and $\alpha_1=1.6$}
\\
  \centering
 $$ \begin{array}{|c|c|c|c|c|c|}
 \hline
 &FS&&FDAC&\\
 \hline
N&\|u-u_h\|&Oder&\|u-u_h\|&Order\\\hline
   2^8& $9.00045e-08 $& & $	9.14455e-08$&\\
  2^9	&$3.47080e-08$&1.37	&$3.61159e-08$&1.34\\
2^{10}&$1.24339e-08$&1.48	&$1.29618e-08$&1.48\\
2^{11}&$4.44658e-09$&1.48&$4.82754e-09$&1.42 \\
2^{12}&$1.75634e-09$&1.34&$1.93550e-09$&1.32\\
2^{13}&$9.95153e-10$&0.82&$1.00011e-09$&0.95\\\hline
%

  \end{array}$$
\end{table}

\begin{table}

\renewcommand{\tablename}
\caption{\centering{ Table 4} \protect \\ \qquad \qquad  CPUs of FS and FDAC for Experiment 2 with $\alpha_0=1.2$ and $\alpha_1=1.6$}
\\
  \centering
 $$ \begin{array}{|c|c|c|c|c|c|}
 \hline
 &FS&FS&FDAC&FDAC\\
 \hline
N&CPU_{M}&CPU_{S}&CPU_{S}&CPU_{M}\\\hline
 2^8& 0.028&0.002 &	0.092&	0.010  \\
 2^9	&0.093&	0.005	&	0.041&	0.014\\
2^{10}	&0.324	&0.024	&	0.087&	0.029\\
2^{11}	&1.28	&0.164&	 0.183&	0.064\\
2^{12}	&5.59&	0.810&0.400&	0.139	\\
2^{13}	&22.52&	3.58	& 0.855&	0.320\\
2^{14}&91.46&	18.8&		1.85	&0.693	\\
2^{15}&	367.38	&98.94	&  	4.12	&1.55	\\
2^{16}&-&-&	8.48	&2.38	\\
2^{17}&-&-&	17.88&	5.52	\\
2^{18}&-&-&	36.94	&12.70	\\\hline


  \end{array}$$
\end{table}

\begin{table}

\renewcommand{\tablename}
\caption{\centering{ Table 5} \protect \\ \qquad \qquad CPUs of FS and FDAC for Experiment 2
with $\alpha_0=1.0$ and $\alpha_1=1.5$}
\\
  \centering
 $$ \begin{array}{|c|c|c|c|c|c|}
  \hline
 &FS&FS&FDAC&FDAC\\
 \hline
N&CPU_{M}&CPU_{S}&CPU_{M}&CPU_{S}\\\hline
 2^8 &0.023&	0.002	&0.037&	0.007\\
 2^9	&0.093	&0.005&0.042&	0.011	\\
 2^{10} &	0.411&	0.025&0.088	&0.029	\\
 2^{11} &	1.33	&0.218&0.184&	0.081\\
2^{12} &	5.06	&0.811&	0.398&	0.226\\
 2^{13} &	20.14	&3.47& 0.853&	0.705\\
 2^{14} &	93.18	&17.83&	1.84&	2.36\\
 2^{15} &	366.62	&99.45&	4.11	&9.19
 \\\hline


  \end{array}$$
\end{table}

\section*{Acknowledgments}
This work was funded by the OSD/ARO MURI Grant W911NF-15-1-0562,
  by the National Science Foundation under grants DMS-1620194 and by Natural Science Foundation of Shandong Province under grants ZR2019BA026.

\end{document}